\documentclass[a4paper,11pt]{amsart}
\usepackage{multirow}

\theoremstyle{plain}

\newtheorem{thm}{Theorem}[section]
\newtheorem{lem}[thm]{Lemma}
\newtheorem{prop}[thm]{Proposition}

\theoremstyle{remark}
\newtheorem{rem}[thm]{Remark}

\theoremstyle{definition}
 
\newtheorem{example}[thm]{Example}

\newcommand{\cC}{\mathcal{C}}

\newcommand{\cQ}{\mathcal{Q}}

\newcommand{\bc}{\mathbb{C}}
\newcommand{\CC}{\mathbb{C}}

\newcommand{\bff}{\mathbb{F}}
\newcommand{\bq}{\mathbb{Q}}
\newcommand{\bz}{\mathbb{Z}}
\newcommand{\ZZ}{\mathbb{Z}}

\newcommand{\NN}{\mathbb{N}}

\newcommand{\bp}{\mathbb{P}}

\newcommand{\p}{\mathfrak{p}}

\DeclareMathOperator{\coker}{coker}
\DeclareMathOperator{\rank}{rank}
\DeclareMathOperator{\MW}{MW}
\DeclareMathOperator{\Proj}{Proj}
\DeclareMathOperator{\length}{length}
\DeclareMathOperator{\codim}{codim}

\numberwithin{equation}{section}

\title{Mordell-Weil groups and Zariski triples}

\author[J.I.~Cogolludo]{Jos\'e Ignacio Cogolludo-Agust{\'\i}n}
\address{Departamento de Matem\'aticas, IUMA, Facultad de Ciencias\\
Universidad de Zaragoza\\
c/ Pedro Cerbuna 12\\
E-50009 Zaragoza SPAIN}
\email{jicogo@unizar.es}

\author[R.~Kloosterman]{Remke Kloosterman}
\address{Institut f\"ur Mathematik, Humboldt-Universit\"at zu Berlin,
Unter den Linden 6, D-10099 Berlin, Germany} 
\email{klooster@math.hu-berlin.de}

\thanks{The authors would  like to thank the referee for several useful suggestions.
The first author is partially supported by the Spanish Ministry 
of Education MTM2010-21740-C02-02. The second author is partially supported by DFG-grant KL 2244/2-1.}
\subjclass[2000]{14H30, 14J30, 14H50, 11G05, 57M12, 14H52}
\begin{document}

\begin{abstract}
We prove the existence of three irreducible curves $C_{12,m}$ of degree 12 with the 
same number of cusps and different Alexander polynomials. This exhibits a Zariski triple.
Moreover we provide a set of generators for the elliptic threefold with constant 
$j$-invariant 0 and discriminant curve $C_{12,m}$. Finally we consider general degree 
$d$ base change of $C_{12d,m}$ and calculate the dimension of the equisingular deformation space.
\end{abstract}

\maketitle

\section*{Introduction}
In this paper we give an example of a Zariski triple. More concretely we construct three irreducible 
plane curves $C_{12,0}, C_{12,1},C_{12,2}$ of degree 12 with 32 ordinary cusps and no further singularities, 
such that the fundamental groups $\pi_1(\bp^2\setminus C_{12,m})$, $m=0,1,2$ are pairwise non-isomorphic.
In order to show that the fundamental groups are pairwise distinct, we do not calculate the fundamental 
groups themselves, but an invariant associated with the fundamental group, namely the Alexander polynomial.

The Alexander polynomial $\Delta(t)$ of a cuspidal curve $C$ is trivial unless the degree of $C$ is of the 
form $6k$, $k\in \NN$. In the latter case it has the form
$$
\Delta(t)=(t^2-t+1)^m.
$$

For the case $k=1$ the Alexander polynomials of such curves are completely understood, namely 
$$
m=
\begin{cases}
0 & \text{if } \#\Sigma \leq 5\\
\#\Sigma -6 & \text{if } \#\Sigma \geq 7,
\end{cases}
$$
where $\Sigma$ is the set of cusps. However, if $\#\Sigma=6$, $m$ equals either 1 or 0 depending on whether 
the six cusps are on a conic or not. This was noted by Zariski {}\cite{Zariski}. 
The two types of sextics with 6 cusps form a so-called Zariski pair, i.e., the combinatorial data of both 
curves coincides, but their complements are not homeomorphic.

For the case $k=2$ there is not such a simple description of $m$ in terms of the number of cusps.
We prove the following result, exhibiting a Zariski triple:
\begin{thm} 
There exist degree 12 curves $C_{12,m}$ with {}$m\in \{0,1,2\}$ with precisely 32 ordinary cusps and no further 
singularities such that the degree of the Alexander polynomial of $C_{12,m}$ equals $2m$.
\end{thm}
This result is a consequence of Propositions~\ref{propDeg0},~\ref{propDeg2}, and~\ref{propDeg4}.

In the case of degree 12 cuspidal curves it is known that if $C$ has at most 23 cusps, then $\Delta(t)=1$. 
From~\cite{KloSyz} it follows that if $C$ has at most 28 cusps, then $m\leq 1$. However, we are not aware of any 
example of a plane curve with either 29, 30 or 31 cusps such that $m\geq 2$. For this reason degree 12 curves 
with 32 cusps seems to be the easiest instance to exhibit a Zariski triple.

Let $f\in \CC[u,v,w]_{6k}$ be a polynomial. A quasi-toric relation of type $(2,3,6)$ is a triple $(h_1,h_2,h_3)$ 
of polynomials such that $h_1^2+h_2^3+fh_3^6=0$. Quasi-toric relations are in one-to-one correspondence with 
$\CC(u,v)$-rational points on the elliptic curve $y^2=x^3+f(u,v,1)$ over $\CC(u,v)$, and hence in one-to-one 
correspondence with rational sections of the elliptic threefold $y^2=x^3+f$ in $\bp(1,1,1,2k,3k)$. The latter 
two sets have a natural group structure.

For a general $k$, there is a complete description of $m$ as follows. 

\begin{thm}[{\cite{CogLib}}]
\label{thm-mw-qt}
The Alexander polynomial of an irreducible curve $C=\{ f=0\}\subset \bp^2$ with only cusps and nodes as 
singularities is non-trivial if and only if there exist three polynomials $h_i\in \CC[u,v,w]$, $i=1,2,3$ such that
$$
h_1^2 + h_2^3 + f h_3^6 = 0.
$$

I.e., the Mordell-Weil group of the elliptic threefold $y^2=x^3+f$ is not trivial.
Moreover, the rank of the Mordell-Weil group equals $2m$, the degree of $\Delta(t)$.
\end{thm}

In the first two sections we calculate the degree of $\Delta(t)$. For this we use the following result due to Zariski: 

\begin{thm}[Zariski]
The degree of the Alexander polynomial of an irreducible cuspidal curve of degree $6k$, with cuspidal locus $\Sigma$, 
equals
\begin{equation}
\label{eqnCok} 
2\dim \coker \left(\bc[u,v,w]_{5k-3} {}\stackrel{\oplus ev_p}{\longrightarrow}  \oplus_{p \in \Sigma} \bc\right).
\end{equation}
\end{thm}

Using this result it is easy to construct an example with 32 cusps and $\deg \Delta(t)\geq 4$, namely we 
pull back a degree 6 curve with 8 cusps under a general degree 4 map $\kappa_2:\bp^2\to \bp^2$. The obtained curve 
has degree 12 and 32 cusps. As an immediate consequence of Theorem~\ref{thm-mw-qt}, the degree of the Alexander 
polynomial cannot decrease under base change, therefore, we obtain a degree 12 cuspidal curve $C_{12,2}$ with 32 
cusps whose Alexander polynomial has degree at least~4.

To find a Zariski triple it suffices then to construct two degree 12 cuspidal curves with 32 cusps whose Alexander 
polynomial has degree 0 and 2. The construction of both these curves is more involved. We start with a sextic 
$C_{6,6}$ with 6 cusps not on a conic. Then we take two very special degree four covers $\bp^2\to \bp^2$ such 
that the inverse image of $C_{6,6}$ has precisely 32 cusps. 
To show that the degree of the Alexander polynomial is either 0 or 2 one has to calculate (\ref{eqnCok}). In one 
of the two examples this can be done directly by using a computer algebra package. This yields an example  where 
the  degree of the Alexander polynomial is two. In the second example this approach is not feasible since the 
minimal field of definition of the ideal of the cusps is too large. 

To actually compute the degree of the Alexander polynomial we have first to reduce the curve modulo a prime $\p$ of 
$\overline{\bq}$, so that the ideal of the cusps can be defined over the prime field $\bff_p$ of the residue field 
of $\p$. We show that in our example this is the case for a prime $\p$ lying over $457$. We then show that in our 
case reducing modulo $\p$ only increases  the dimension of (\ref{eqnCok}).  We calculate the dimension of the 
co-kernel and finally obtain that the degree of the Alexander polynomial is zero. It turns out that for our particular 
example the smallest suitable prime $\p$ is 457.
 
In the third section we give generators for the Mordell-Weil group of $y^2=x^3+f$. 
By~\cite{CogLib} the rank of this group equals the degree of $\Delta(t)$.

Let $C_{12d,m}$ be the pullback of $C_{12,m}$ under a general degree $d$ map $\kappa:\bp^2\to\bp^2$. 
Then $C_{12d,m}$ has $32d^2$ cusps. Let $\mathcal{C}_{12d,m}$ be the equisingular deformation space of $C_{12d,m}$. 
In the final section we show that
\begin{thm} The codimension of $\mathcal{C}_{12d,m}$ in $\CC[u,v,w]_{12d}$ equals
\[ 64d^2-m(2d-1)(d-1).\]
\end{thm}
In particular, if $m>0$ and $d>1$ then the codimension is smaller than expected.

\section{Construction of the curves}
In this section we construct three curves of degree 12 with exactly 32 ordinary cusps. We start by constructing 
three configurations of curves each involving one sextic and three lines. We obtain the four degree 12 curves as 
degree four covers of these sextics ramified along the lines.

Consider the following two sextic curves: 
\begin{enumerate}
 \item\label{66} A sextic $\cC_{6,6}$ with exactly six cusps and whose inflexion points are such that there exist 
two bitangents $\ell_1$ and $\ell_2$ each one intersecting $\cC_{6,6}$ exactly at two inflexion points.
 \item\label{68} A sextic $\cC_{6,8}$ with exactly eight cusps.
\end{enumerate}

The existence of $\cC_{6,8}$ can be found in~\cite[Table 1, nt15]{oka-pho-class}. For the existence of $\cC_{6,6}$
consider the (smooth) Fermat cubic $u^3+v^3+w^3$. Let $\xi$ be a primitive sixth root of unity. The inflexion points and 
tangencies of this cubic are shown in Table~\ref{tabInFlex}. Namely, the points $p_{i,j}$ denote the 9 inflexion points
of the Fermat cubic. The lines $t_{i,j}$ are tangent to such a cubic at $p_{i,j}$. Finally, the rows are arranged so
that the lines $t_{i,1}, t_{i,2}, t_{i,3}$ are concurrent for each $i=1,2,3$. The first column describes the intersection
of such concurrent tangent lines.

\begin{table}[h]\label{tabInFlex}
\caption{} 
\centering 
{\small{
\begin{tabular}{|c|c||c||c|}
\hline

\multirow{2}{*}{$\cap t_{1,j}=\{[0:1:0]\}$} 
&$p_{1,1}:=[1:0:-1]$ & $p_{1,2}:=[\xi
:0:-1]$ & $p_{1,3}:=[\xi^2:0:-1]$
\\
\cline{2-4}
&$t_{1,1}:=\{u+w\}$ & $t_{1,2}:=\{u+\xi w\}$ & $t_{1,3}:=\{u+\xi^2 w\}$
\\
\hline\hline

\multirow{2}{*}{$\cap t_{2,j}=\{[0:0:1]\}$} 
&$p_{2,1}:=[1:-1:0]$ & $p_{2,2}:=[\xi:-1:0]$ & $p_{2,3}:=[\xi^2:-1:0]$
\\
\cline{2-4}
&$t_{2,1}:=\{u+v\}$ & $t_{2,2}:=\{u+\xi v\}$ & $t_{2,3}:=\{u+\xi^2 v\}$
\\
\hline\hline

\multirow{2}{*}{$\cap t_{3,j}=\{[1:0:0]\}$} 
&$p_{3,1}:=[0:1:-1]$ & $p_{3,2}:=[0:\xi:-1]$ & $p_{3,3}:=[0:\xi^2:-1]$
\\
\cline{2-4}
&$t_{3,1}:=\{v+w\}$ & $t_{3,2}:=\{v+\xi w\}$ & $t_{3,3}:=\{v+\xi^2 w\}$
\\
\hline

\end{tabular}
}}
\label{tab:x3y3z3}
\end{table}

Consider the Kummer cover $\kappa_2$ of order 2 ramified along $t_{1,1}$, $t_{1,2}$, and $t_{2,1}$,
that is, $[u:v:w]\mapsto [t_{1,1}^2:t_{1,2}^2:t_{2,1}^2]$. The preimage of $\cC_3$ under $\kappa_2$
is $\cC_{6,6}$ a sextic with six cusps which are the preimages of the inflexion points $P_{1,1}$,
$P_{1,2}$, and $P_{2,1}$. Since $t_{1,1}$, $t_{1,2}$, and $t_{1,3}$ are concurrent lines at a point
$[0:1:0]$ which is totally ramified (i.e. it has only one preimage), the preimage of $t_{1,3}$ 
decomposes in a product of two lines, say $\ell_1$ and $\ell_2$. Also note that $\ell_1$ and $\ell_2$ 
are bitangent lines through the inflexion points in the preimage of $P_{1,3}$.

\subsection{Construction of $C_{12,0}$}\label{subsec-661}
Consider the curve $\cC_{6,6}$ given in~\eqref{66}. The line $\ell_1$ is a bitangent through two 
inflexion points. A straightforward calculation shows that the six cusps of the sextic $\cC_{6,6}$ do 
not lie on a conic. Moreover, the combinatorics of $\cC_{6,6} \cup \ell_1$ cannot be obtained using any 
sextic of torus type (that is, with six cusps on a conic).

The curve $\cC_{6,6}$ contains 20 inflexion points, four of which (the preimages of $P_{1,3}$) 
belong to the bitangents $\ell_1$ and $\ell_2$. Out of the remaining 16 inflexion points select 
two and their tangent lines, say $t_1$ and $t_2$. Now consider a Kummer covering of order 2 ramified 
along $\ell_1$, $t_1$, and $t_2$. Note that the pull-back of $\cC_{6,6}$ by such a covering is a curve 
of degree 12, say $\cC_{12,0}$, with 32 cusps (24 of which come from the preimages of the 6 cusps and 
the remaining 8 are the preimages of the inflexion points which are tangent to the ramification 
lines $\ell_1$ and $t_1$, and $t_2$).

\subsection{Construction of $C_{12,1}$}\label{subsec-662}
Now consider a Kummer covering of order 2 ramified along $\ell_1$, $\ell_2$, and a generic line.
Note that the pull-back of $\cC_{6,6}$ by such a covering is a curve of degree 12, say $\cC_{12,1}$, 
with 32 cusps (24 of which come from the preimages of the 6 cusps and the remaining 8 are the 
preimages of the inflexion points which are tangent to the ramification lines $\ell_1$ and~$\ell_2$).

\subsection{Construction of $C_{12,2}$}\label{subsec-68}
Described in~\cite[Table 1, nt15]{oka-pho-class}, the sextic $\cC_{6,8}$ is a torus type curve with 
6 cusps on a conic. Therefore, by the argument mentioned above, it cannot contain a bitangent to 
two inflexion points. Such a curve must be self-dual and hence it contains exactly 8 inflexion points.

Consider a Kummer covering of order 2 ramified along three generic lines. Note that the pull-back 
of $\cC_{6,8}$ by such a covering is a curve of degree 12, say $\cC_{12,2}$, with 32 cusps (coming 
from the preimages of the 8 cusps).

\section{Calculation of the degree of the Alexander polynomials}
Let $S=\bc[u,v,w]$. Let $f\in S_{6k}$ and denote  $C=V(f(u,v,w))\subset \bp^2$. By \cite{CogLib} the degree of the 
Alexander polynomial of $C$ equals the Mordell-Weil rank of the elliptic 
threefold $y^2=x^3+f$. In this section we discuss a method to calculate the Mordell-Weil rank of such threefolds, 
following~\cite{KloSyz}.
Assume that 
$C$ is a reduced curve with only ordinary cusps as singularities. Let $\Sigma$ be the set of cusps of $C$, 
and $I$ the ideal of $\Sigma$. Zariski \cite{Zariski} proved that if $C$ is an irreducible cuspidal plane curve 
then the degree of the Alexander 
polynomial of $C$ (equivalently the Mordell-Weil rank of $y^2=x^3+f$) equals
\begin{equation}
\label{rankFormula} 
2\dim \coker \left( S_{5k-3}\stackrel{\oplus ev_p}{\longrightarrow} \oplus_{p\in \Sigma} \bc\right).
\end{equation}

Equivalently, we have that the Mordell-Weil rank equals twice $h_I(5k-3)-\#\Sigma$, where $h_I$ is the Hilbert 
function of $I$. One can express $h_I(5k-3)$ in terms of the degrees of the generators and syzygies of $I$. 
This is done in \cite{KloSyz}. In some cases the description in terms of syzygies is more useful than the 
description in terms of linear systems. In \cite[Lemma 4.3]{KloSyz} the following is proved:
\begin{prop}\label{prpRnkSyz}
Let
\[ 0 \to \oplus_{i=1}^t S(-b_i) \to  \oplus_{i=1}^{t+1} S(-a_i) \to S \to S/I\to 0 \]
be a minimal free resolution of $I$. Then
\begin{enumerate}
\item $\sum b_i=\sum a_i$
\item $2\# \Sigma=\sum b_i^2-\sum a_i^2$.
\item $b_i\leq 5k$ for all $i$.
\item $\#\Sigma \leq 3k\min(a_i)$.
\end{enumerate}
In  particular,
\[ \rank \MW(\pi)=2\#\{i \mid b_i=5k\}.\]

Moreover, we can permute the $a_i$ and the $b_i$ such that
\begin{enumerate}
\item $b_1\geq b_2\geq\dots\geq b_t$.
\item $a_1\geq a_2\geq \dots\geq a_{t+1}$.
\item $a_i<b_i$ for all $i\in\{1,\dots,t\}$.
\end{enumerate}
\end{prop}

{} For later use, we will now apply  this proposition to restrict the possible resolutions of $I$ in the case of 
a sextic with $8$ cusps:
\begin{example}\label{ExaDeg6Alex4}
Suppose $k=1$ and $\#\Sigma=8$.
We want to determine all possibilities for $a_i$ and $b_i$. 

{From} $8=\#\Sigma \leq 3ka_{t+1}=3a_{t+1}$ it follows that $a_{t+1}\geq 3$. 

Let $r$ be the number of $b_i$'s that are equal to 5. For $j=3,4$ let $A_j$ be the number of $b_i$ there are equal 
to $j$ minus  the number of $a_i$ that are equal to $j$. One obtains the following three equalities:
\[ r+A_4+A_3+1=0,\;\; 5r+4A_4+3A_3=0,\;\; 25r+16A_4+9A_3=16.\]
{From} this it follows that $r=2$, $A_4=-1$ and $A_3=-2$.

{From} Proposition~\ref{prpRnkSyz} we have the following two possibilities: 
$b_1=b_2=5, a_1=4,a_2=a_3=3$ and $b_1=b_2=5,b_3=4,a_1=a_2=4,a_3=a_4=3$.

The latter possibility can be ruled out. Suppose we have a relation in degree 4 between the two generators of degree 3. 
For instance, if $f_1$ and $f_2$ are distinct generators of $I$ of degree 3, then there exist linear forms $g_1$ and $g_2$ 
such that $g_1f_1=g_2f_2$. This implies that $f_1$ and $f_2$ have a common factor of degree 2. So $f_1=g_2h$ 
and $f_2=g_1h$. Only one of the cusps can lie on the intersection of $g_1$ and $g_2$. Hence there are at least 
7 points that are both cusps of $C$ and points on the conic $h$. By Bezout's theorem this implies that the 
conic $V(h)$ is a component of the sextic $C$, i.e., we can write $C$ as a union of a conic and a quartic curve.
A conic cannot have a cusp as singularity, a quartic curve can have at most 3 cusps. Hence $C$ has at most 4 cusps, 
contradicting $\#\Sigma=8$. 

So only the case $b_1=b_2=5, a_1=4,a_2=a_3=3$ might occur.
\end{example}

We use the above results to calculate the degree of the Alexander polynomial of $C_{12,0}$, $C_{12,1}$ and $C_{12,2}$. 
We will show below that the degrees are $0,2,4$. Since all three curves have 32 cusps,  it follows that any two of 
them form a Zariski pair and the three of them form a Zariski triple.

\subsection{Alexander polynomial of $C_{12,0}$} 
\label{subsec-alex-120}
The calculation of the degree of the Alexander polynomial of the curve $C_{12,0}$ is more difficult than it is in 
the other two cases $C_{12,1}$ and $C_{12,2}$. The main problem is that $L'$, the minimal field of definition of $C_{12,0}$, 
is much bigger than that of either $C_{12,1}$ or $C_{12,2}$. A computer algebra package is used to calculate the value of 
the Hilbert function of the ideal of cusps of $C_{12,0}$ and $C_{12,1}$ at $5k-3$. It turns out that it is not feasible for 
$C_{12,0}$. However, by reducing the curve modulo a prime $\p$ of $L'$ lying over the prime $457$ the Hilbert function of 
the ideal of the cusps of~$C_{12,0}$ can be determined.

\begin{lem}\label{lemRedEqn} The curve $\overline{C_{12,0}}$ in $\bp^2_{\bff_{457}}$ defined by the vanishing of  
\[ \begin{array}{c}409u^8v^2w^2+32u^6v^2w^4+203u^6v^4w^2+263u^4v^2w^6+224u^4v^6w^2+\\290u^4v^4w^4+85u^{12}+ 
160w^{12}+317u^2w^6v^4+220u^2w^2v^8+436u^{10}v^2+\\276u^8v^4+399u^6v^6+82u^{10}w^2+352u^4w^8+
318u^8w^4+198u^6w^6+\\31u^2w^{10}+210u^2w^8v^2+451u^2w^4v^6+121w^4v^8+ 
306w^8v^4+291w^6v^6+\\31w^{10}v^2+208u^2v^{10}+103u^4v^8+148v^{12}+325v^{10}w^2\end{array}\]
has 32 cusps. 

Moreover, the lines $\ell_1$ and $\ell_2$ in the construction of $C_{12,0}$ can be chosen so that 
$\overline{C_{12,0}}$ is the reduction of $C_{12,0}$ modulo a prime ideal $\p$ in $\overline{\bq}$ lying over $457$ and 
each of the 32 cusps of $\overline{C_{12,0}}$ is the reduction modulo $\p$ of a cusp of~$C_{12,0}$.
\end{lem}

\begin{proof}
Let $\p$ be a prime of $L=\bq(\xi)$ lying over $457$.
Since $457$ equals $1 \bmod 3$, the field $\bff_{457}$ contains a primitive third root of unity. In particular, 
$\p$ is split, i.e. $O_L/\p\cong \bff_{457}$. Since $C_{6,6}$ is defined over $L=\bq(\xi)$, it follows that the 
reduction $\overline{C_{6,6}}$ of $C_{6,6}$ modulo $\p$ is defined over $\bff_{457}$. A straightforward computation 
shows that the prime $457$ splits completely in the field $L_2$, the minimal field over which the line $\ell_1$ is 
defined. Let $\overline{\ell_1}$ be the reduction of $\ell_1$ modulo $457$.

Consider now the Hessian $\overline{H}$ of $\overline{C_{6,6}}$ and intersect  $\overline{H}$ with $\overline{C_{6,6}}$. 
This intersection has two $\bff_{457}$-rational points that correspond to flex lines, which are not bitangents. Call 
the corresponding flex lines $\overline{t_1},\overline{t_2}$. An easy computation shows that the reduction modulo 
$\p$ of the (scheme-theoretic) intersection of  the Hessian of $C_{6,6}$ with ${C_{6,6}}$ is precisely the 
(scheme-theoretic) intersection of $\overline{C_{6,6}}$ with $\overline{H}$.
In particular, the lines $\overline{t_1},\overline{t_2}$ are reductions modulo $\p$ of the flex lines $t_1$ and $t_2$ of 
$C_{6,6}$. Let $\overline{C_{12,0}}$ be the pullback under the Kummer map of order 2, ramified along $\overline{t_1}$, 
$\overline{t_2}$ and $\overline{\ell_1}$. Then $\overline{C_{12,0}}$ is the reduction of $C_{12,0}$ modulo a prime of 
$\bq$ lying over~$457$.

By making the above construction explicit one easily shows that the equation mentioned above is an equation for 
$\overline{C_{12,0}}$.

The  curve $\overline{C_{12,0}}$ has 32 cusps for the same reason $C_{12,0}$ does.
{From} the geometric description of the position of the cusps it follows immediately that all of the 32 cusps of $C_{12,0}$ 
reduce to different cusps of $\overline{C_{12,0}}$. (cf. Subsection~\ref{subsec-661}.)
\end{proof}

\begin{lem} Let $\tilde{I}\subset\bff_{457}[u,v,w]$ be the ideal of the cusps of {}$\overline{C_{12,0}}$. Then 
$\dim \tilde{I}_7=4$.
\end{lem}

\begin{proof} We use Singular \cite{DGPS} to compute a resolution of $\tilde{I}$. Let $\tilde{S}$ be the ring
$\bff_{457}[u,v,w]$. We obtain the following resolution of $\tilde{I}$:
\[ 0 \to \tilde{S}(-9)^4\to \tilde{S}(-8)\oplus \tilde{S}(-7)^4 \to\tilde{S}\to \tilde{S}/\tilde{I} \to 0.\]
Hence $\dim \tilde{I}_7=4$.
\end{proof}

\begin{prop} \label{propDeg0} 
In the construction of $C_{12,0}$ the two flex lines $t_1$ and $t_2$ can be chosen so that the degree of the 
Alexander polynomial of $C_{12,0}$ equals~0.
\end{prop}

\begin{proof}
Since $\#\Sigma=32$ and $\dim S_7=36$ it follows from (\ref{rankFormula})  that the degree of the Alexander polynomial 
equals $2(\dim_\bc I_7-4)$. Hence in order to show that the rank is zero it suffices to prove that $\dim I_7\leq 4$.

Choose the lines $t_1$ and $t_2$ as described in the proof of Lemma~\ref{lemRedEqn}. In particular,  there exists a 
prime $\p$ of $\overline{\bq}$ over $457$  such that reduction modulo $\p$ of $C_{12,0}$ is the curve 
{}$\overline{C_{12,0}}$ of Lemma~\ref{lemRedEqn}.

Let $\overline{I_7}$ be the reduction of $I_7$ modulo $\p$. (The reduction of $I_7$ as a vector space.)
One can easily show that $I_7$ has a $\ZZ_{\p}$-integral basis such that modulo $\p$ the elements of this basis remain 
linearly independent. Hence the $\bff_{457}$-dimension of $\overline{I_7}$ equals the $\bc$-dimension of $I_7$. 

Now $I_7$ consists of polynomials that vanish at all of the cusps of $C_{12,0}$. This implies that the elements of 
$\overline{I_7}$ vanish at the reduction modulo $\p$ of all of the cusps of $C_{12,0}$, which implies
\[ \overline{I_7} \subset \tilde{I}_7.\]
In particular, we have
\[\dim I_7=\dim \overline{I_7} \leq \dim \tilde{I}_7=4.\]
\end{proof}

\begin{rem} 
We would like  to identify the smallest prime number $p$ such that the reduction of $C_{12,0}$ modulo a prime 
$\p$ over $p$ can be defined over the prime field $\bff_p$. 
For this we need the reduction of both $C_{6,6}$ and the union of the three lines $\ell_1$, $t_1$ and $t_2$ 
to be defined over $\bff_p$. We determine the smallest prime $p$ where a slightly  stronger condition holds, 
namely each of the lines $\ell_1, t_1$ and $t_2$ can be defined over the prime field.

Now $C_{6,6}$ is defined over $\bff_p$ if and only if $p\equiv 1 \bmod 3$ and the line $\ell_1$ is defined over 
$\bff_p$ if and only if $p\equiv 1 \bmod 12$. Then we look for the smallest prime $p\equiv 1 \bmod 12$ 
such that the Hessian of $C_{6,6}$ and $C_{6,6}$ intersect in a $\bff_p$-rational intersection point, which is 
not a bitangent. It turns out that this prime is $p=457$.
\end{rem}

\subsection{Alexander polynomial of $C_{12,1}$}

\begin{prop} \label{propDeg2} 
The degree of the Alexander polynomial of $C_{12,1}$ is two.
\end{prop}

\begin{proof}
The first step in the construction is to take a Kummer cover $\kappa_2$ of order 2 ramified along 
$t_{1,1},t_{1,2}$ and $t_{2,1}$. We choose the coordinates  on the domain of $\kappa_2$ such that the 
preimage of $t_{1,1}$ is $u=0$, the preimage of $t_{1,2}$ is $v=0$, and the preimage of $t_{2,1}$ is $w=0$.
We obtain then the following equation for $C_{6,6}$ {}(where $\xi$ is a primitive sixth root of unity):
\begin{eqnarray*} 
u^6+3u^4v^2-3u^2v^4+v^6-(3{\xi}^2-3)u^4w^2-(12{\xi}^2+6)u^2v^2w^2&&\\ 
-(3{\xi}^2+6)v^4w^2+9{\xi}^2 u^2w^4+(9{\xi}^2+9)v^2w^4-(6{\xi}^2-3)w^6.
\end{eqnarray*}
This curve has precisely 6 cusps.

The preimage of $t_{1,3}$ under this map  consists of the two lines $\ell_1,\ell_2$ given by 
$u\pm i(\zeta_3+1)v=0$. In the construction of $C_{12,1}$ one needs to choose a general line $\ell_3$, i.e. 
$\ell_3$ intersects $C_{6,6}$ in six distinct points, and none of these six points is on $\ell_1\cup \ell_2$. 
The curve $C_{12,2}$ depends on the choice of the general line, but the Alexander polynomial does not depend on this choice.
This follows from the fact that by taking a different line we obtain a curve which 
is an equisingular deformation of the original curve. Since the Alexander polynomial is invariant under equisingular 
deformations the choice of this line does not influence the degree of the Alexander polynomial.
We pick the line $5u+w=0$ for $\ell_3$.

Consider the Kummer cover of {}order 2 of $\bp^2$ ramified along the lines $\ell_1,\ell_2,\ell_3$. Call this 
curve $C_{12,1}$. The singular points of $C_{12,1}$ are either in the inverse image of the singular points 
of $C_{6,6}$ or lie on the ramification divisor of the cover. If $P$ is in the preimage of a singular point, 
then the Kummer map is unramified at $P$. Hence the Kummer map is locally an isomorphism and therefore the 
type of singularity does not change. This yields 24 cusps. 

Now suppose that $P$ lies on the ramification divisor. If $P$ is in the preimage of  $\ell_1$ or $\ell_2$, then 
the image of $P$ is  an inflexion point of $C_{6,6}$ and $\ell_i$ is a flex line. An easy computation in 
local coordinates shows that $P$ must be a cusp of $C_{12,0}$. This yields 8 cusps. 

If $P$ is  in the preimage of $\ell_3$, then at the image of $P$ the curve $C_{6,6}$ intersects $\ell_3$ 
transversely. An easy computation shows that $P$ is a smooth point of $C_{12,0}$. 
 
In total we find 32 cusps $P_1,\dots,P_{32}$. The computer algebra package Singular~\cite{DGPS} was used to 
compute the following resolution of the ideal $I$ of $\{P_1,\dots,P_{32}\}$:
\[ 0 \to S(-10)\oplus S(-9)^2 \to S(-8)\oplus {}S(-7)^2\oplus S(-6)\to S \to S/I \to 0.\]
Since exactly one of the $b_i$'s equals 10, we obtain from Proposition~\ref{prpRnkSyz} that the degree of the 
Alexander polynomial is~2.
\end{proof}

\subsection{Alexander polynomial of $C_{12,2}$}
An example where the Mordell-Weil rank equals 4 is easy to construct and does not require a computer algebra package.

\begin{prop}\label{propDeg4} 
The degree of the Alexander polynomial of $C_{12,2}$ equals~4.
\end{prop}
\begin{proof}
We start by considering the locus of the cusps of $C_{6,8}$. From Example~\ref{ExaDeg6Alex4} it 
follows that a minimal resolution of the ideal of the cusps of $C_{6,8}$ has the following form
\[ 0 \to S(-5)^2\to S(-4)\oplus S(-3)^2\to S \to S/I(\Sigma) \to 0.\]

Take a general  Kummer cover $\kappa_2$ of {}order 2, i.e. assume that the branch locus of $\kappa_2$ 
intersects the curve $C_{6,8}$ transversely. Let $C_{12,2}$ be the inverse image of $C_{6,8}$. 
Let $\tilde{\Sigma}$ be the set of points where $C_{12,2}$ has cusps. Since none of the cusps of $C_{6,6}$ 
lie on the critical locus and $C_{6,6}$ intersects the critical locus transversely, it follows that 
$\tilde{\Sigma}=\varphi^{-1}(\Sigma)$ and that $I(\tilde{\Sigma})=\varphi^*(I(\Sigma))$. In particular, 
the minimal resolution of $I(\Sigma)$ can be pulled back to a minimal resolution of $I(\tilde{\Sigma})$:
\[ 0 \to S(-10)^2\to S(-8)\oplus S(-6)^2\to S \to S/I(\tilde{\Sigma}) \to 0.\]
By Proposition~\ref{prpRnkSyz} the degree of the Alexander polynomial of $C_{12,2}$ equals twice the number 
of independent syzygies of degree 10, hence the Alexander polynomial of  $C_{12,2}$ has degree~4.
\end{proof}

\section{Description of Mordell-Weil groups}
Let $ f\in S_{6k}$. The Mordell-Weil group of the elliptic threefold $y^2=x^3+ f$ over $\bp^2$ is the group 
of $\CC(u,v)$-valued points of the elliptic curve $y^2=x^3+f(u,v,1)$ over $\CC(u,v)$. Such a point is of the 
form $(\frac{h_2}{h_3^2},\frac{h_1}{h_3^2})$ for  some $h_i\in \CC[u,v]$. Hence elements of the Mordell-Weil 
group correspond with elements of 
$$
\cQ_{(2,3,6)}(f):=\{(h_1,h_2,h_3)\in \CC[u,v,w] \mid h_1^2+h_2^3+h_3^6f=0\},
$$
i.e. they correspond with quasi-toric relations of $f$ of elliptic type $(2,3,6)$. The Mordell-Weil group is 
finitely generated if and only if $f$ is not a sixth power in $\CC(u,v)$.

According to~\cite{CogLib}, if $f$ is an irreducible polynomial such that $C=V(f)$ has only cusps
and nodes as singularities, then the rank of the group of quasi-toric relations of $f$ equals the degree of 
the Alexander polynomial $\Delta(t)$ of the curve $C=\{f=0\}$.

According to subsection~\ref{subsec-alex-120} the equation of $C_{12,0}$ given in Lemma~\ref{lemRedEqn} does not 
decompose into any (non-trivial) quasi-toric relation of elliptic type $(2,3,6)$. Hence the Mordell-Weil 
group is trivial in this case.

The elliptic curve $y^2=x^3+1$ has complex multiplication by a sixth root of unity. Therefore the group of 
rational sections is a $\ZZ[\xi]$-module. We will now give a generator of the free part of the Mordell-Weil 
group of $C_{12,1}$ (as a $\ZZ[\xi]$-module).
{From} the construction given in section~\ref{subsec-662}, note that the Fermat curve satisfies
\begin{equation}
\label{eq-121-1}
t_{1,1}t_{1,2}t_{1,3}+v^3=u^3+v^3+w^3
\end{equation}
After the Kummer cover $\kappa_2$ of order two ramified along $t_{1,1}$, $t_{1,2}$, and $t_{2,1}$, 
identity~\eqref{eq-121-1} becomes

\begin{equation}
\label{eq-121-2}
(t_{1,1}t_{1,2})^2\ell_1\ell_2+q^3=f_{6,6}
\end{equation}
where by abuse of notation $t_{1,1}$ (resp. $t_{1,2}$) denotes the set-theoretically preimage of $t_{1,1}$
(resp. $t_{1,2}$) by $\kappa_2$, $f_{6,6}$ denotes $\kappa_2^* (u^3+v^3+w^3)$ which is the equation
of $ C_{6,6}$, and $q$ denotes the conic which is preimage of $v=0$ by $\kappa_2$.
Finally, after the second Kummer cover of order 2, ramified along $\ell_1$, $\ell_2$, and a generic line,
say $\ell$, note that identity~\eqref{eq-121-2} becomes

\begin{equation}
\label{eq-121-3}
(\tilde t_{1,1}\tilde t_{1,2}\ell_1\ell_2)^2+\tilde q^3=\tilde f_{6,6}
\end{equation}
where by abuse of notation $\ell_1$ (resp. $\ell_2$) denotes the set-theoretically preimage of $\ell_1$
(resp. $\ell_2$) by the cover and the tildes denote the transformation of a given equation by the cover.
Note that identity~\eqref{eq-121-3} generates $\cQ_{(2,3,6)}(C_{12,1})$. This shows the following.

\begin{prop}
\label{prop-q121}
The free part of the Mordell-Weil group $\cQ_{(2,3,6)}(C_{12,1})$ is generated as a $\ZZ[\xi]$-module by
$(\tilde t_{1,1}\tilde t_{1,2}\ell_1\ell_2,\tilde q,-1)$.
\end{prop}

\begin{proof}
On the one hand $h:=(\tilde t_{1,1}\tilde t_{1,2}\ell_1\ell_2,\tilde q,-1)$ satisfies an equation of elliptic
type $(2,3,6)$ for $C_{12,1}$ as shown in~\eqref{eq-121-3}, therefore $h\in \cQ_{(2,3,6)}(C_{12,1})$. 
Consider $\omega_6$ the order 6 action $(x,y)\mapsto (\xi^2 x,-y)$ associated with the elliptic curve $y^2+x^3+1=0$ 
over $\CC[u,v]$ and define the map $\varphi:\bz[\xi] \to \cQ_{(2,3,6)}(C_{12,1})$ given by $a+b \xi \mapsto (a+ b \omega) h$.
Note that $\varphi$ is clearly injective. Moreover, $\rank \bz[\xi]=\rank \cQ_{(2,3,6)}(C_{12,1})=2$ by 
Proposition~\ref{propDeg2}. Hence, $\varphi$ is an isomorphism if and only if $h=\mu \tilde h$ implies $\mu$ is
a unit in $\bz[\xi]$ (that is, its module as a complex number is one). This is a consequence of the fact that the pencil 
generated by $(\tilde t_{1,1}\tilde t_{1,2}\ell_1\ell_2)^2$ and $\tilde q^3$ is primitive (see~\cite[Proposition~4.4]{CogLib}).
\end{proof}

\begin{rem}
Note that $\cQ_{(2,3,6)}(C_{12,1})$ is strictly bigger than $\cQ_{(2,3,6)}(C_{6,6})=\{0\}$, therefore
the rank of  the Mordell-Weil group after a base change increases.
\end{rem}

The curve $C_{12,2}$ requires special attention: 

Using the techniques in~\cite{swbmn} one can show that the cusps of $C_{6,8}$ can be grouped in four pairs such that
there exist four conics $q_1,...,q_4$ satisfying the following conditions:
\begin{enumerate}
\item each pair of cusps belongs to three of the four conics,
\item any two of these pairs of cusps are the intersection points of exactly two conics, and
\item the remaining two pairs of cusps belong to the remaining two conics.
\end{enumerate}
In the pencil generated by $f_{6,8}$ and each of these conics $q_i$ there exists a cubic $g_i$ such that
\begin{equation}
\label{eq-122-1}
g_i^2+q_i^3=f_{6,8}
\end{equation}
for $i=1,2,3,4$.
After a Kummer covering of order two ramified along three generic lines, one obtains identities
\begin{equation}
\label{eq-122-2}
\tilde g_i^2+\tilde q_i^3=f_{12,2}
\end{equation}
for $i=1,2,3,4$.

\begin{prop}
The Mordell-Weil group $\cQ_{(2,3,6)}(C_{12,2})$ is generated (as a $\ZZ$-module) by $(\tilde g_i,\tilde q_i,-1)$, $i=1,2,3,4$.
\end{prop}

\begin{proof}
As in Proposition~\ref{prop-q121}, it is a consequence of the fact that the pencils 
described in~\eqref{eq-122-1} are primitive.
\end{proof}

\section{Deformations}

Let $C=V(f)$ be a cuspidal curve. Let $J$ denote the saturation of the Jacobian ideal of $C$, and let $I$ denote 
the ideal of the cusps of $C$. In particular, $\sqrt{J}=I$ and $\length \Proj S/J=2\length \Proj S/I$.

For $m\in \{0,1,2\}$ let $C_{12,m}$ as before. Let $d>1$ be an integer. Let $\kappa:\bp^2\to\bp^2$ be a 
general degree $d$ map. Let $C_{12d,m}$ be the pullback of $C_{12,m}$. Then $C_{12d,m}$ has $32d^2$ cusps.

Consider now the space $\mathcal{C}_{12d,m}$ of equisingular deformations of $C_{12d,m}$ in $\bp^2$. The 
tangent space to $\mathcal{C}_{12d,m}$ is given by $J_{12d}(C_{12d,m})$ the degree $12d$-part of $J(C_{12d,m})$~\cite{GL}. 
The dimension of this space can be read off from the minimal resolution of $J(C_{12d,m})$. Since $\kappa$ 
is general, this resolution is just the pullback of the minimal resolution of~$J(C_{12,m)}$.

To study the resolution of $J(C_{12d,m})$ note the following:
\begin{lem} Let $J$ be $J(C_{12d,m})$. Let $W$ be the hypersurface $y^2=x^3+\kappa^*(f)$ in $\bp(1,1,1,4d,6d)$ Then
\[ 2\#\Sigma-\dim (S/J)_{14d-3} = \frac{1}{2}\rank \MW(\pi).\]
 
\end{lem}
\begin{proof}
From \cite{ell3HK} it follows that the Mordell-Weil rank of $y^2=x^3+f$ equals $h^4(W)-1$. Each cusp of $C$ 
yields a $D_4$-singularity of $W$. Following Rams \cite{Rams} we obtain 
\[h^4(W)-1 = \dim \coker \left(x S_{10d-3}\oplus S_{14d-3} \to \oplus_{p\in \Sigma} \bc^3\right).\]
Moreover $xf\oplus g$ is mapped to $(f(p),g(p),\frac{d}{d\ell} g(p))$, where $\ell^2$ is the generator of the 
ideal of the tangent cone of $C$ at $p$. In particular, we have that $h^4(W)-1$ equals
\[ \#\Sigma -\dim(S/I)_{10d-3}+2\# \Sigma-\dim (S/J)_{14d-3}.\]
Since $h^4(W)-1=\rank \MW(\pi)$ and $\frac{1}{2}\rank \MW(\pi)=\# \Sigma-\dim(S/I)_{10d-3}$, by~(\ref{eqnCok}) 
it follows that $2\#\Sigma-\dim (S/J)_{14d-3}=\frac{1}{2}\rank \MW(\pi)$.
\end{proof}

\begin{lem} The minimal resolution of $J^{(m)}:=J(C_{12,m})$ has only syzygies and generators of degree up to 14. 
None of the generators of syzygies has degree 13, and there are precisely $m$ syzygies of degree~14.
\end{lem}
\begin{proof}
Combining the previous lemma with the argument used in the proof of~\cite[Proposition 3.3]{KloSyz} one obtains 
that if at least one of the syzygies or generators had degree greater than 14, then the Mordell-Weil rank of 
$C_{12d,m}$ would be asymptotically of the form $cd^2$, for some positive constant $c$. This is impossible by 
either Libgober's global divisibility theorem for Alexander polynomials or by the Shioda-Tate formula for 
elliptic surfaces (cf. the proof of~\cite[Proposition 3.3]{KloSyz}). 

Hence all syzygies are of degree at most 14 and all generators are of degree at most 13. Similarly, as in the 
proof of \cite[Proposition 3.3]{KloSyz}, one obtains that the number of syzygies of maximal degree, i.e. degree 14 equals
\[ 2\# \Sigma -\dim (S/J^{(m)})_{11}=\frac{1}{2}\rank \MW(\pi)=m.\]

The claim for degree 13 requires that each $m$ be done separately. If $m=2$, then the resolution of 
$J^{(m)}$ is the pullback of a resolution under a degree 2 cover. Hence all the generators and syzygies have even degrees.

For the case $m=0,1$ we compute a resolution of $J$ using Singular~\cite{DGPS}. 
It turns out that none of the generators or syzygies has degree~13.
\end{proof}

\begin{prop}\label{prpTDim} The tangent space of $\mathcal{C}_{12d,m}$ has codimension
 \[ 2\# \Sigma -m(2d-1)(d-1) \]
in $S_{12d}$.
\end{prop}
\begin{proof}
The tangent space to $\mathcal{C}_{12d,m}$ is given by $J_{12d}(C_{12d,m})$, the degree $12d$-part of $J(C_{12d,m})$. 
A minimal resolution of $J(C_{12d,m})$ can be pulled back from a minimal resolution of $J(C_{12,m})$. From the 
previous lemma it follows now that a minimal resolution of $J(C_{12d,m})$ has $m$ syzygies of degree $14d$ and no 
further syzygies or generators of degree at least~$12d+1$.

A standard exercise in commutative algebra expresses the difference between the length of $\Proj S/J(C_{12d,m})$ 
and $\codim J_{12d}(C_{12d,m})$ in terms of the generators and syzygies of degree greater than $12d$. From this we obtain that
\[ \codim J_{12d}(C_{12d,m})-\length(\Proj S/J(C_{12d,m})= m(2d-1)(d-1).\]
\end{proof}

To calculate the codimension of $\mathcal{C}_{12d,m}$ we treat the case $m=2$ separately.

\begin{lem}\label{lemdimcountDeg4} Consider the locus  
\[ \{f\in S_{6k} \mid V(F) \mbox{ has } 8k^2 \mbox{ cusps and } \rank \MW(\pi)=4\}.\]
Then this locus contains a component of codimension at most $16k^2-(k-1)(k-2)$.
\end{lem}

\begin{proof} Fix polynomials $f_1\in S_{2k}$ and $f_2,f_3\in S_k$. Let $\eta$ be a square root of $-3$. Define
 \[\begin{array}{c} w_1:=f_1-f_2f_3,\; w_2:=f_1+f_2f_3, \\
v_1:=\frac{1}{2}f_2f_3^2-\eta f_1f_2-f_2^2f_3-\frac{\eta}{2}f_1f_3, \\
v_2:=\frac{1}{2}f_2f_3^2+\eta f_1f_2+f_2^2f_3-\frac{\eta}{2}f_1f_3.\end{array}\]
Set $f:=v_1^2-w_1^3$. By construction $f=v_2^2-w_2^3$. Hence for general $f_1,f_2,f_3$, the points $P_i=(w_i,v_i)$ 
yield different points on the elliptic curve given by $v^2=w^3+f$. One can actually show that the 
$P_1,P_2,\xi^* P_1, \xi^* P_2$ generate a subgroup of rank 4, where $\xi^*(w,v)=(\xi w,-v)$ and $\xi$ is a 
primitive sixth root of unity (this fact is not required for the proof).

If the $f_i$'s are chosen to be generic, then $\Sigma$ contains $V(v_1,w_1) \cup V(v_2,w_2)$. This locus can also be described as
\[ V(f_1,f_2f_3) \cup V(f_1-f_2f_3, 4f_2+(\eta+1)f_3) \cup V(f_1+f_2f_3, 4f_2+(\eta-1)f_3).\]
For generic $f_i$ this defines $8k^2$ points.

We claim that for generic $f_i$'s the curve $V(f)$ has precisely $8k^2$ cusps. Since the Milnor number of a cusp is 2, 
it suffices to give an example where the length of $S/J(f)$ is $16k^2$. If $k=1$ this can be done by taking 
$f_1=x,f_2=y$ and $f_0=x^2+y^2+z^2$. For a general $k$ we can take a generic degree $k$ base change of the $k=1$ 
example and obtain a curve with precisely $8k^2$ cusps.

From the above description of the locus of the cusps, it follows that $\Sigma$ is the union of a complete 
intersection of two degree $2k$ curves say $Q_1'=f_1$ and $Q_2'=f_2f_3$ and two complete intersections of two degree $k$ curves $T_i=R_i=0$, with
$T_i=Q_1'+(-1)^i Q_2'$ and $R_i=4f_2+(\eta-(-1)^i)f_3)$. Hence the ideal of $\Sigma$ is generated by $T_1T_2,R_1T_2,R_2T_1$, 
i.e. the generators have degrees $4k,3k,3k$. There are two obvious syzygies of degree $5k$. Hence the resolution 
of the ideal of the cusps equals 
\[ 0 \to S(-5k)^2\to S(-4k)\oplus S(-3k)^2\to S\to S/I\to 0.\]
In particular, we have an example of a degree $6k$ curve with $8k^2$ cusps and $\Delta(t)=(t^2-t+1)^2$.

We want to calculate the dimension of the space of degree $6k$ curves of the above form. The dimension of 
$S_{2k}\times S_k\times S_k$ equals $3k^2+6k+3$. This quantity can be rewritten as $\dim S_{6k}-16k^2+(k-1)(k-2)$.

Consider now the map $S_{2k}\times S_k\times S_k \to S_{6k}\times S_{2k}\times S_{2k} \to S_{6k}$. The first map 
is defined by $(f_0,f_1,f_2)\mapsto (f,f_0-f_1f_2,f_0+f_1f_2)$ and the second map is the projection on the first factor. 
The fiber over a general point $P\in S_{6k}\times S_{2k}\times S_{2k}$ in the image of the first map is finite since
we can obtain $f_0$ and $f_1f_2$ from $P$, which leaves only finitely many possibilities for $f_1$ and~$f_2$.

Take now a point $f\in S_{6k}$ in the image of the composition. We claim that either there are only finitely many 
points in the image of the first map mapping to $f$ or $f$ is of the form $h^6$ with $h\in S_k$. Indeed, suppose 
that the fiber of $f$ intersected with the image of the first map has a positive dimensional component. Then we 
have a map from a complex curve to the locus $L:=\{(P,Q)\in S_{2k}\times S_{3k} \mid Q^2-P^3=F\}$. Hence we would 
have a map from this curve to the Mordell-Weil group of $y^2=x^3+f$. It is well known that this group is finitely 
generated if and only if $f$ is not a sixth power in $S_{\bullet}$. In the former case there are at most countably 
many points in $L$, and therefore $L$ cannot contain a complex curve. In particular the general fiber of the 
composition $S_{2k}\times S_k \times S_k \to S_{6k}$ is finite.
\end{proof}

\begin{prop} The codimension of $\mathcal{C}_{12d,m}$ in $S_{12d}$ equals
 \[ 2\# \Sigma -m(2d-1)(d-1) .\]
\end{prop}

\begin{proof} By Proposition~\ref{prpTDim} we know that the codimension is at least $ 2\# \Sigma -m(2d-1)(d-1)$. 
Requiring a cusp yields two conditions on a polynomial. Hence we know that $\codim \mathcal{C}_{12d,m}\leq 2 |\Sigma|$. 
This finishes the case $m=0$.

Suppose now that $m=1$. Then from Proposition~\ref{prop-q121} it follows that $k^*(f)=g^2+h^3$ for some $g\in S_{6d}, h\in S_{4d}$.
Now $\dim S_{12d}-\dim S_{4d}-\dim S_{6d}=46d^2+3d-1$. The map $S_{4d}\to S_{6d}\to S_{12d}$ mapping $(h,g)$ to 
$g^2+h^3$ has finite fibers. This follows from the same  argument as in the final paragraph of the proof of 
Lemma~\ref{lemdimcountDeg4}. A general element of this locus has $24d^2$ cusps. Hence any component of the locus 
of curves of the form $g^2+h^3$ having $32d^2$ cusps has codimension at most $46d^2+3d-1+16d^2=2\# \Sigma-(2d^2-3d+1)$, 
which finishes the case~$m=1$.

Suppose now that $m=2$. Then $C_{12d,m}$ lies on the component constructed in the proof of Lemma~\ref{lemdimcountDeg4}. 
This finishes the case~$m=2$.
\end{proof}

\end{document}